\documentclass[12pt,reqno]{amsart}

\usepackage[dvips]{graphicx}
\usepackage{amssymb, latexsym, amsmath, amscd, array, amsthm
}
\usepackage[all,cmtip]{xy}

\newtheorem{theorem}{Theorem}[section]
\newtheorem{lemma}[theorem]{Lemma}
\newtheorem{proposition}[theorem]{Proposition}
\newtheorem{corollary}[theorem]{Corollary}

\theoremstyle{definition}
\newtheorem{definition}[theorem]{Definition}

\newtheorem{remark}[theorem]{Remark}

\numberwithin{equation}{section}

\def\co{\colon\thinspace}

\newcommand\R {{\mathbb R}} 
\newcommand\Z {{\mathbb Z}}

\newcommand\sys{{\rm Sys}} 
\newcommand\area{{\rm Area}} 
\newcommand\vol{{\rm Vol}}
\newcommand\length{{\rm length}}

\newcommand{\cf}{{\it cf.}}
\newcommand{\wh}{\widehat}

\long\def\forget#1\forgotten{}

\begin{document}
\title{Relative systoles of relative-essential~$2$--complexes}

\author[K. Katz]{Karin Usadi Katz}

\author[M.~Katz]{Mikhail G. Katz$^{*}$}

\thanks{$^{*}$Supported by the Binational Science Foundation (grant
2006393)}

\address{Department of Mathematics, Bar Ilan University, Ramat Gan
52900 Israel} 

\email{katzmik@macs.biu.ac.il}

\author[S.~Sabourau]{St\'ephane Sabourau}

\address{Laboratoire de Math\'ematiques et Physique Th\'eorique,
Universit\'e Fran\c{c}ois-Rabelais Tours, F\'ed\'eration Denis Poisson -- CNRS,
Parc de Grandmont, 37200 Tours, France}

\email{sabourau@lmpt.univ-tours.fr}

\author[S.~Shnider]{Steven Shnider}

\address{Department of Mathematics, Bar Ilan University, Ramat Gan
52900 Israel}

\email{shnider@macs.biu.ac.il}

\author[S.~Weinberger]{Shmuel Weinberger$^{**}$}

\thanks{$^{**}$Partially supported by NSF grant DMS 0504721 and the
BSF (grant 2006393)}

\address{Department of Mathematics, University of Chicago, Chicago, IL
60637}

\email{shmuel@math.uchicago.edu}



\date{\today}

\begin{abstract}
We prove a systolic inequality for a~$\phi$--relative systole of
a~$\phi$--essential~$2$--complex~$X$, where~$\phi \co \pi_1(X) \to G$
is a homomorphism to a finitely presented group~$G$.  Thus, we show
that universally for any~$\phi$--essential
Riemannian~$2$--complex~$X$, and any~$G$, the following inequality is
satisfied:~$\sys(X, \phi)^2 \leq 8 \, \area(X)$.  Combining our
results with a method of L.~Guth, we obtain new quantitative results
for certain~$3$--manifolds: in particular for the Poincar\'e homology
sphere~$\Sigma$, we have~$\sys(\Sigma)^3 \leq 24 \, \vol(\Sigma)$.
\end{abstract}

\maketitle 

\tableofcontents

\forget
Nothing appears in the text!
\forgotten

\section{Relative systoles}

We prove a systolic inequality for a~$\phi$--relative systole of a
\mbox{$\phi$--essential}
$2$--complex~$X$, where~$\phi \co \pi_1(X) \to G$ is a homomorphism to a
finitely presented group~$G$.  Thus, we show that universally for
any~$\phi$--essential Riemannian~$2$--complex~$X$, and any~$G$, we
have~$\sys(X, \phi)^2 \leq 8 \, \area(X)$.  Combining our results with
a method of L.~Guth, we obtain new quantitative results for
certain~$3$--manifolds: in particular for the Poincar\'e
homology sphere~$\Sigma$, we have~$\sys(\Sigma)^3 \leq 24 \, \vol(\Sigma)$.  To
state the results more precisely, we need the following definition.

Let~$X$ be a finite connected~$2$--complex.  Let~$\phi \co \pi_{1}(X) \to
G$ be a group homomorphism.  Recall that~$\phi$ induces a classifying
map (defined up to homotopy)~$X \to K(G,1)$.

\begin{definition}
The complex~$X$ is called~$\phi$--{\em essential\/} if the classifying
map~$X \to K(G,1)$ cannot be homotoped into the~$1$--skeleton
of~$K(G,1)$.
\end{definition}


\begin{definition}
Given a piecewise smooth Riemannian metric on~$X$, the~$\phi$--relative
systole of~$X$, denoted~$\sys(X,\phi)$, is the least length of a loop
of~$X$ whose free homotopy class is mapped by~$\phi$ to a nontrivial
class.
\end{definition}

When~$\phi$ is the identity homomorphism of the fundamental group, the
relative systole is simply called the systole, and denoted~$\sys(X)$.

\begin{definition} 
\label{def:sigma}
The~$\phi$--systolic area~$\sigma_\phi(X)$ of~$X$ is defined as
\begin{equation*}
\sigma_{\phi}(X) = \frac{\area(X)}{\sys(X,\phi)^{2}}.
\end{equation*}
Furthermore, we set
\begin{equation*}
\sigma_{*}(G) = \inf_{X, \phi} \sigma_{\phi}(X),
\end{equation*}
where the infimum is over all~$\phi$--essential piecewise Riemannian
finite connected \mbox{$2$--complexes}~$X$, and homomorphisms~$\phi$
with values in~$G$.
\end{definition}

In the present text, we prove a systolic inequality for
the~$\phi$--relative systole of a~$\phi$--essential~$2$--complex~$X$.
More precisely, in the spirit of Guth's text \cite{Gu09}, we prove a
stronger, {\em local\/} version of such an inequality, for almost
extremal complexes with minimal first Betti number.  Namely, if~$X$
has a minimal first Betti number among all~$\phi$--essential piecewise
Riemannian \mbox{$2$--complexes} satisfying~$\sigma_{\phi}(X) \leq
\sigma_*(G) +\varepsilon$ for an~$\varepsilon>0$, then the area of a
suitable disk of~$X$ is comparable to the area of a Euclidean disk of
the same radius, in the sense of the following result.

\begin{theorem} 
\label{13}
Let~$\varepsilon >0$.  Suppose~$X$ has a minimal first Betti number
among all~$\phi$--essential piecewise Riemannian
$2$--complexes satisfying~$\sigma_{\phi}(X) \leq \sigma_*(G)
+\varepsilon$.  Then each ball centered at a point~$x$ on
a~$\phi$--systolic loop in~$X$ satisfies the area lower bound
\begin{equation*}
\area \, B(x,r) \geq
\frac{\left(r-\varepsilon^{1/3}\right)^2}{2+\varepsilon^{1/3}}
\end{equation*}
whenever~$r$ satisfies~$\varepsilon^{1/3} \leq r \leq
\frac{1}{2}\sys(X,\phi)$.
\end{theorem}

A more detailed statement appears in~Proposition~\ref{prop:minB}.  The
theorem immediately implies the following systolic inequality.

\begin{corollary} 
\label{coro:A}
Every finitely presented group~$G$ satisfies
\begin{equation*}
\sigma_*(G) \geq \frac{1}{8}, 
\end{equation*}
so that every piecewise Riemannian~$\phi$--essential~$2$--complex~$X$
satisfies the inequality
\begin{equation*}
\sys(X,\phi)^{2} \leq 8 \, \area(X).
\end{equation*}
\end{corollary}

In the case of the absolute systole, we prove a similar lower bound
with a Euclidean exponent for the area of a suitable disk, when the
radius is smaller than half the systole, without the assumption of
near-minimality.  Namely, we will prove the following theorem.

\begin{theorem} 
\label{theo:B}
Every piecewise Riemannian essential~$2$--complex~$X$ admits a
point~$x\in X$ such that the area of the~$r$--ball centered at~$x$ is
at least~$r^2$, that is,
\begin{equation}
\label{11c}
\area ( B(x,r)) \geq r^2,
\end{equation}
for all~$r \leq \frac{1}{2} \sys(X)$.
\end{theorem}

We conjecture a bound analogous to \eqref{11c} for the area of a
suitable disk of a \mbox{$\phi$--essential}~$2$--complex~$X$, with
the~$\phi$--relative systole replacing the systole, \cf~the GG-property
below.  The application we have in mind is in the case
when~$\phi \co \pi_1(X)\to \Z_p$ is a homomorphism from the fundamental
group of~$X$ to a finite cyclic group.  Note that the conjecture is
true in the case when~$\phi$ is a homomorphism to~$\Z_2$, by Guth's
result \cite{Gu09}.

\begin{definition}[GG-property%
\footnote{GG-property stands for the property analyzed by M.~Gromov
and L.~Guth}
] 
\label{def:GG}
%
%
Let~$C>0$.  Let~$X$ be a finite connected~$2$--complex,
and~$\phi \co \pi_1(X) \to G$, a group homomorphism.  We say that~$X$ has
the~$\rm{GG}_{C}$-property for~$\phi$ if
%
%
every piecewise smooth Riemannian metric on~$X$ admits a point~$x \in
X$ such that the~$r$--ball of~$X$ centered at~$x$ satisfies the bound
\begin{equation} 
\label{eq:ball}
\area \, B(x,r) \geq C r^2,
\end{equation}
for every~$r \leq \frac{1}{2} \sys(X,\phi)$.
\end{definition}

Note that if the~$2$--complex~$X$ is~$\varepsilon$--almost minimal,
i.e., satisfies the bound $\sigma_{\phi}(X) \leq G_*(G) +
\varepsilon$, and has least first Betti number among all such
complexes, then it satisfies~\eqref{eq:ball} for some~$C>0$ and
for~$r\geq \varepsilon^{1/3}$ by Theorem~\ref{13}.

Modulo such a conjectured bound, we prove a systolic inequality for
closed~$3$--manifolds with finite fundamental group.

\begin{theorem}
\label{theo:main}
Let~$p\geq 2$ be a prime.  Assume that
every~$\phi$--essential~$2$--complex has the~$\rm{GG}_C$-property
\eqref{eq:ball} for each homomorphism~$\phi$ into~$\Z_p$ and for some
universal constant~$C>0$.  Then every orientable closed
Riemannian~$3$--manifold~$M$ with finite fundamental group of order
divisible by~$p$, satisfies the bound
\begin{equation*}
\sys(M)^3 \leq 24 \, C^{-1} \; \vol(M).
\end{equation*}
More precisely, there is a point~$x\in M$ such that the volume of
every~$r$--ball centered at~$x$ is at least~$\frac{C}{3}r^{3}$, for all
$r \leq \frac{1}{2} \sys(M)$.
\end{theorem}


A slightly weaker bound can be obtained modulo a weaker GG-property,
where the point~$x$ is allowed to depend on the radius~$r$.

Since the GG-property is available for~$p=2$ and~$C=1$ by Guth's article
\cite{Gu09}, we obtain the following corollary.

\begin{corollary}
Every closed Riemannian~$3$--manifold~$M$ with fundamental group of
even order satisfies
\begin{equation}
\label{Poincare}
\sys(M)^3 \leq 24 \; \vol(M).
\end{equation}
\end{corollary}

For example, the Poincar\'e homology~$3$--sphere satisfies the systolic
inequality \eqref{Poincare}. \\

In the next section, we present related developments in systolic
geometry and compare some of our arguments in the proof of
Theorem~\ref{theo:main} to Guth's in~\cite{Gu09},
\cf~Remark~\ref{rem:compare}.  Additional recent developments in
systolic geomety include \cite{AK, BB10, Bal08, e7, BT, Be08, Bru,
Bru2, Bru3, DKR, DR09, Elm10, EL, Gu09, KK, KK2, Ka4, KR2, KSh, NR,
Par10, Ro, RS08, Sa08, Sa10}.

\section{Recent progress on Gromov's inequality}

M.~Gromov's upper bound for the~$1$--systole of an essential
manifold~$M$ \cite{Gr1} is a central result of systolic geometry.
Gromov's proof exploits the Kuratowski imbedding of~$M$ in the Banach
space~$L^\infty$ of bounded functions on~$M$.  A complete analytic
proof of Gromov's inequality \cite{Gr1}, but still using the
Kuratowski imbedding in~$L^\infty$, was recently developed by
L.~Ambrosio and the second-named author~\cite{AK}.  See
also~\cite{AW}.

S.~Wenger~\cite{wen} gave a complete analytic proof of an
isoperimetric inequality between the volume of a manifold~$M$, and its
filling volume, a result of considerable independent interest.  On the
other hand, his result does not directly improve or simplify the proof
of Gromov's main filling inequality for the filling radius.  Note that
both the filling inequality and the isoperimetric inequality are
proved simultaneously by Gromov, so that proving the isoperimetric
inequality by an independent technique does not directly simplify the
proof of either the filling radius inequality, or the systolic
inequality.

L.~Guth \cite{Gu11} gave a new proof of Gromov's systolic inequality
in a strengthened {\em local\/} form.  Namely, he proved Gromov's
conjecture that every essential manifold with unit systole contains a
ball of unit radius with volume uniformly bounded away from zero.

Most recently, Guth \cite {Gu09} re-proved a significant case of
Gromov's systolic inequality \cite{Gr1} for essential manifolds,
without using Gromov's filling invariants.

Actually, in the case of surfaces, Gromov himself had proved better
estimates, without using filling invariants, by sharpening a technique
independently due to Y.~Burago and V.~Zalgaller \cite[p.~43]{BZ}, and
J.~Hebda \cite{Hebda}.  Here the essential idea is the following.

Let~$\gamma(s)$ be a minimizing non-contractible closed geodesic of
length~$L$ in a surface~$S$, where the arclength parameter~$s$ varies
through the interval~$[-\frac{L}{2}, \frac{L}{2}]$.  We consider
metric balls (metric disks)~$B(p,r) \subset S$ of radius~$r<
\frac{L}{2}$ centered at~$p=\gamma(0)$.  The two points~$\gamma(r)$
and~$\gamma(-r)$ lie on the boundary sphere (boundary curve)~$\partial
B(p,r)$ of the disk.  If the points lie in a common connected
component of the boundary (which is necessarily the case if~$S$ is a
surface and~$L=\sys(S)$, but may fail if~$S$ is a more
general~$2$--complex), then the boundary curve has length at
least~$2r$.  Applying the coarea formula
\begin{equation}
\label{11b}
\area \, B(p,r)=\int_0^r \length \, \partial B(p,\rho) \, d\rho,
\end{equation}
we obtain a lower bound for the area which is quadratic in~$r$.

Guth's idea is essentially a higher-dimensional analogue of Hebda's,
where the minimizing geodesic is replaced by a minimizing
hypersurface.  Some of Guth's ideas go back to the even earlier texts
by Schoen and Yau \cite{SY78, SY79}.

The case handled in \cite{Gu09} is that of~$n$--dimensional manifolds
of maximal~$\Z_2$--cuplength, namely~$n$.  Thus, Guth's theorem covers
both tori and real projective spaces, directly generalizing the
systolic inequalities of Loewner and Pu, see \cite{Pu} and \cite{SGT}
for details.


\begin{remark} \label{rem:compare}
To compare Guth's argument in his text~\cite{Gu09} and our proof of
Theorem~\ref{theo:main}, we observe that the topological ingredient of
Guth's technique exploits the multiplicative structure of the
cohomology ring~$H^*(\Z_2;\Z_2)=H^*(\R {\mathbb P}^\infty; \Z_2)$.
This ring is generated by the~$1$--dimensional class.  Thus, every
$n$--dimensional cohomology class decomposes into the cup product
of~$1$--dimensional classes.  This feature enables a proof by
induction on~$n$.

Meanwhile, for~$p$ odd, the cohomology ring~$H^*(\Z_p;\Z_p)$ is not
generated by the~$1$--dimensional class; see Proposition~\ref{42} for
a description of its structure.  Actually, the square of
the~$1$--dimensional class is zero, which seems to yield no useful
geometric information.

Another crucial topological tool used in the proof of~\cite{Gu09} is
Poincar\'e duality which can be applied to the manifolds representing
the homology classes in~$H_*(\Z_2;\Z_2)$.  For~$p$ odd, the homology
classes of~$H_{2k}(\Z_p;\Z_p)$ cannot be represented by manifolds.
One could use D.~Sullivan's notion of~$\Z_p$--manifolds,
\cf~\cite{Su,MS}, to represent these homology class, but they do not
satisfy Poincar\'e duality.

Finally, we mention that, when working with cycles representing
homology classes with torsion coefficients in~$\Z_p$, we exploit a
notion of volume which ignores the multiplicities in~$\Z_p$,
\cf~Definition~\ref{def:Vol}.  This is a crucial feature in our proof. 
%
%
Note that minimal cycles with torsion coefficients were studied by
B.~White \cite{Wh2}.
\end{remark}

\section{Area of balls in~$2$--complexes}

It was proved in~\cite{Gr1} and~\cite{KRS} that a finite~$2$--complex
admits a systolic inequality if and only if its fundamental group is
nonfree, or equivalently, if it is~$\phi$--essential for~$\phi= {\rm
Id}$.

In~\cite{KRS}, we used an argument by contradiction, relying on an
invariant called {\em tree energy\/}, to prove a bound for the
systolic ratio of a~$2$--complex.  We present an alternative short proof
which yields a stronger result and simplifies the original argument.

\begin{theorem} 
\label{theo:r2}
Let~$X$ be a piecewise Riemannian finite essential~$2$--complex.  There
exists~$x \in X$ such that the area of every~$r$--ball centered at~$x$
is at least~$r^2$ for every~$r \leq \frac{1}{2} \sys(X)$.
\end{theorem}

As mentioned in the introduction, we conjecture that this result still
holds for~$\phi$--essential complexes and with the~$\phi$--relative
systole in place of~$\sys$.

\begin{proof}
We can write the Grushko decomposition of the fundamental group of~$X$
as
\begin{equation*}
\pi_1(X) = G_1*\cdots*G_r*F,
\end{equation*}
where~$F$ is free, while each group~$G_i$ is nontrivial,
non-isomorphic to~$\Z$, and not decomposable as a nontrivial free
product.

Consider the equivalence class~$[G_1]$ of~$G_1$ under external
conjugation in~$\pi_1(X)$.  Let~$\gamma$ be a loop of least length
representing a nontrivial class~$[\gamma]$ in~$[G_1]$.  Fix~$x \in
\gamma$ and a copy of~$G_1 \subset \pi_1(X,x)$ containing the homotopy
class of~$\gamma$.  Let~$\overline{X}$ be the cover of~$X$ with
fundamental group~$G_1$.

\begin{lemma}
We have~$\sys(\overline{X}) = \length(\gamma)$.
\end{lemma}

\begin{proof}
%
The loop~$\gamma$ lifts to~$\overline{X}$ by construction of the subgroup
$G_1$.  Thus,~$\sys(\overline{X}) \leq \length(\gamma)$.  Now, the
cover~$\overline{X}$ does not contain noncontractible loops~$\delta$
shorter than~$\gamma$, because such loops would project to~$X$ so that
the nontrivial class~$[\delta]$ maps into~$[G_1]$, contradicting our
choice of~$\gamma$.
\end{proof}

Continuing with the proof of the theorem, let~$\bar{x} \in \overline{X}$
be a lift of~$x$.  Consider the level curves of the distance function
from~$\bar{x}$.  Note that such curves are necessarily connected, for
otherwise one could split off a free-product-factor~$\Z$
in~$\pi_1(\overline{X})=G_1$, \cf~\cite[Proposition 7.5]{KRS},
contradicting our choice of~$G_1$.  In particular, the
points~$\gamma(r)$ and~$\gamma(-r)$ can be joined by a path contained
in the curve at level~$r$.  Applying the coarea formula~\eqref{11b},
we obtain a lower bound~$\area \, B(\bar{x},r)\geq r^2$ for the area
of an~$r$--ball~$B(\bar{x},r) \subset \overline{X}$, for all~$r \leq
\frac{1}{2} \length(\gamma)=\frac{1}{2} \sys(\overline{X})$.

If, in addition, we have~$r \leq \frac{1}{2}\sys(X)$ (which apriori
might be smaller than~$\frac{1}{2} \sys(\overline{X})$), then the ball
projects injectively to~$X$, proving that
\begin{equation*}
\area(B(x,r)\subset X) \geq r^2
\end{equation*}
for all~$r\leq \frac{1}{2}\sys(X)$.
\end{proof}

\section{Outline of argument for relative systole}
\label{three}

Let~$X$ be a piecewise Riemannian connected~$2$--complex, and
assume~$X$ is~$\phi$--essential for a group homomorphism
$\phi \co \pi_1(X)\to G$.  We would like to prove an area lower bound
for~$X$, in terms of the~$\phi$--relative systole as in
Theorem~\ref{theo:r2}.  Let~$x \in X$.  Denote by~$B=B(x,r)$ and
$S=S(x,r)$ the open ball and the sphere (level curve) of radius~$r$
centered at~$x$ with~$r < \frac{1}{2} \sys(X,\phi)$.  Consider the
interval~$I=[0,\frac{L}{2}]$, where~$L=\length(S)$.

\begin{definition} 
\label{def:Y}
We consider the complement~$X \setminus B$, and attach to it a buffer
cylinder along each connected component~$S_i$ of~$S$.  Here a buffer
cylinder with base~$S_i$ is the quotient
\begin{equation*}
S_i \times I/ \!\! \sim
\end{equation*}
where the relation~$\sim$ collapses each subset~$S_i \times \{ 0 \}$
to a point~$x_i$.  We thus obtain the space
\[
\left( S_i \times I/ \!\! \sim \right) \cup_f \left( X \setminus B
\right),
\]
where the attaching map~$f$ identifies
$S_i\times\left\{\tfrac{L}{2}\right\}$ with~$S_i\subset X\setminus B$.
To ensure the connectedness of the resulting space, we attach a
cone~$CA$ over the set of points~$A=\{ x_i \}$.  We set the length of the
edges of the cone~$CA$ equal to~$\sys(X,\phi)$.  We will denote by
\begin{equation} \label{eq:Y}
Y=Y(x,r)
\end{equation}
the resulting~$2$--complex.  The natural metrics on~$X \setminus B$ and
on the buffer cylinders induce a metric on~$Y$.
\end{definition}

In the next section, we will show that~$Y$ is~$\psi$--essential for
some homomorphism~$\psi \co \pi_1(Y) \to G$ derived from~$\phi$.  The
purpose of the buffer cylinder is to ensure that the relative systole
of~$Y$ is at least as large as the relative systole of~$X$.  Note that
the area of the buffer cylinder is~$L^2/2$.

We normalize~$X$ to unit relative systole and take a point~$x$ on a
relative systolic loop of~$X$.
%
%
Suppose~$X$ has a minimal first Betti number among the complexes
essential in~$K(G,1)$ with almost minimal systolic area (up to
epsilon).  We sketch below the proof of the local relative systolic
inequality satisfied by~$X$.

If for every~$r$, the space~$Y=Y(x,r)$ has a greater area than~$X$,
then
\begin{equation*}
\area \, B(r) \leq \tfrac{1}{2}(\length \, S(r))^2
\end{equation*}
for every~$r < \frac{1}{2} \sys(X,\phi)$.  Using the coarea
inequality, this leads to the differential inequality~$y(r) \leq
\tfrac{1}{2} y'(r)^2$.  Integrating this relation shows that the area
of~$B(r)$ is at least~$\frac{r^2}{2}$, and the conclusion follows.

If for some~$r$, the space~$Y$ has a smaller area than~$X$, we argue
by contradiction.  We show that a~$\phi$--relative systolic loop of~$X$
(passing through~$x$) meets at least two connected components of the
level curve~$S(r)$.  These two connected components project to two
endpoints of the cone~$CA$ connected by an arc of~$Y \setminus CA$.
%
%
Under this condition, we can remove an edge~$e$ from~$CA$ so that the
space~$Y'=Y \setminus e$ has a smaller first Betti number than~$X$.
Here~$Y'$ is still essential in~$K(G,1)$, and its relative systolic
area is better than the relative systolic area of~$X$, contradicting
the definition of~$X$.

\section{First Betti number and essentialness of~$Y$} \label{sec:remov}

Let~$G$ be a fixed finitely presented group.  We are mostly interested
in the case of a finite group~$G=\Z_p$.  Unless specified otherwise,
all group homomorphisms have values in~$G$, and all complexes are
assumed to be finite.  Consider a homomorphism~$\phi \co \pi_1(X) \to
G$ from the fundamental group of a piecewise Riemannian finite
connected~$2$--complex~$X$ to~$G$.

\begin{definition}
A loop~$\gamma$ in~$X$ is said to be~$\phi$--contractible if the image
of the homotopy class of~$\gamma$ by~$\phi$ is trivial, and
$\phi$--noncontractible otherwise.  Thus, the~$\phi$--systole of~$X$,
denoted by~$\sys(X,\phi)$, is defined as the least length of a
$\phi$--noncontractible loop in~$X$.  Similarly, the
\mbox{$\phi$--systole} based at a point~$x$ of~$X$, denoted
by~$\sys(X,\phi,x)$, is defined as the least length of
a~$\phi$--noncontractible loop based at~$x$.
\end{definition}

\forget
\begin{definition}
\label{42b}
The~$\phi$--systolic area of~$X$ is defined as
$$
\sigma_{\phi}(X) = \frac{\area(X)}{\sys(X,\phi)^{2}}.
$$
\end{definition}
\forgotten

The following elementary result will be used repeatedly in the sequel.

\begin{lemma} 
\label{lem:trivial}
If~$r < \frac{1}{2} \sys(X,\phi,x)$, then
the~$\pi_{1}$--homomorphism~$i_{*}$ induced by the inclusion~$B(x,r)
\subset X$ is trivial when composed with~$\phi$, that is~$\phi \circ
i_{*}=0$.  More specifically, every loop in~$B(x,r)$ is homotopic to a
composition of loops based at~$x$ of length at most~$2r+\varepsilon$, for
every~$\varepsilon>0$.
\end{lemma}

Without loss of generality, we may assume that the piecewise
Riemannian metric on~$X$ is piecewise flat.  Let~$x_{0} \in X$.  The
piecewise flat~$2$--complex~$X$ can be embedded into some~$\R^N$ as a
semialgebraic set and the distance function~$f$ from~$x_0$ is a
continuous semialgebraic function on~$X$, \cf~\cite{BCR98}.
Thus,~$(X,B)$ is a CW-pair when~$B$ is a ball centered at~$x_0$ (see
also \cite[Corollary~6.8]{KRS}).  Furthermore, for almost every~$r$,
there exists a~$\eta >0$ such that the set
\begin{equation*}
\{ x \in X \mid r-\eta < f(x) < r+\eta \}
\end{equation*}
is homeomorphic to~$S(x_0,r) \times (r-\eta,r+\eta)$ where~$S(x_0,r)$
is the~$r$--sphere centered at~$x_{0}$ and the~$t$--level curve of~$f$
corresponds to~$S(x_0,r) \times \{t\}$, \cf~\cite[\S~9.3]{BCR98}
and~\cite{KRS} for a precise description of level curves on~$X$.  
In such case, we say that~$r$ is a \emph{regular value} of~$f$. \\

\forget 

Since the function~$\ell(r) = \length \, f^{-1}(r)$ is piecewise
continuous, \cf~\cite[\S~9.3]{BCR98}, the condition~$\area \, B >
\lambda \, (\length \, S)^{2}$ is open (see ~\eqref{eq:lambda} below).
Therefore, slightly changing the value of~$r$ if necessary, we can
assume that~$r$ is regular. 

\forgotten

Consider the connected~$2$--complex~$Y=Y(x_0,r)$ introduced in
Definition~\ref{def:Y}, with~$r < \frac{1}{2} \sys(X,\phi)$ and~$r$
regular.  Since~$r$ is a regular value, there exists~$r_- \in (0,r)$
such that~$B \setminus B(x_0,r_-)$ is homeomorphic to the product
\[
S \times [r_-,r) = \coprod_i S_i \times [r_-,r).
\]
Consider the map
\begin{equation} 
\label{eq:XY}
\pi \co X \to Y
\end{equation}
which leaves~$X \setminus B$ fixed, takes~$B(x_0,r_-)$ to the vertex
of the cone~$CA$, and sends~$B \setminus B(x_0,r_-)$ to the union of
the buffer cylinders and~$CA$.
This map induces an epimorphism between the first
homology groups.  In particular,
\begin{equation} \label{eq:b1}
b_1(Y) \leq b_1(X).
\end{equation}

\medskip

\forget
\begin{lemma}
\label{lem:betti}
We have
\[
b_1(Y) \leq b_1(X).
\]
\end{lemma}

\begin{proof}
Since~$r$ is a regular value, there exists~$r_- \in (0,r)$ such that~$B \setminus B(x,r_-)$ is homeomorphic to~$S \times [r_-,r) = \coprod_i S_i \times  [r_-,r)$.
The map~$X \to Z$ which leaves~$X \setminus B$ fixed and takes~$B(x,r_-)$ to the vertex of the cone~$C$ of~$Z$ induces an epimorphism between the first homology groups.
Hence the result.
\end{proof}
\forgotten

\forget
\begin{proof}
Let~$f$ be the distance function from~$x_0$.  It is convenient to
introduce the Reeb space~$\wh{X}$ obtained from~$X$ by collapsing to
points the connected components of the level curves~$f^{-1}(t)$, for
every~$t \in [0,r]$ (level curves for~$t>r$ are unaffected).  Note
that the lower bound for the systole no longer holds for~$\wh{X}$ due
to possible ``shortcuts'' created in the graph~$T\subset\wh{X}$
corresponding to~$t\leq r$.

Since the fibers of the map~$X\to \wh{X}$ are connected, by the
covering homotopy property, we obtain that every closed path in
$\wh{X}$ lifts to a closed path in~$X$, proving the surjectivity of
$\pi_1(X)\to \pi_1(\wh{X})$.

We first assume that~$X\setminus B$ is connected.  Then the Reeb
space~$\wh{X}$ is homotopy equivalent%
\footnote{The fact that the ``Reeb graph'' is indeed a finite graph
follows from semialgebraicity; see \cite{KRS} for a detailed
discussion.}
to the union~$Y \cup T$ obtained by attaching a finite graph~$T$ to
the finite set~$\{x_1,\ldots,x_n\} \subset Y$,
cf.~Definition~\eqref{def:Y}.  By van Kampen's theorem, the removal of
the graph leads to a further decrease in the Betti number.  The
non-closed path~$\alpha$ closes up to a loop in~$X$ but not in~$Y$.

If~$X\setminus B$ is not connected, our space~$Y$ is homotopy
equivalent to a connected component of~$\wh{X}$ with the graph~$T$
removed, proving the lemma.
\end{proof}
\forgotten

\forget
Let~$A=\{x_1,\ldots, x_n\}$ be the finite set formed by the
points~$x_i$.  Let~$Y \cup CA$ be the space obtained by attaching a
cone over~$A$ to~$Y$.  Consider the map
\[
\wh{X} \to Y \cup CA
\]
which leaves~$Y$ fixed and takes~$T \setminus (\cup_i e_i)$ to the
vertex of~$CA$, where the~$e_i$ are the semi-edges of~$T$ with
endpoints~$x_i$.  The composite
\[
X \to \wh{X} \to Y \cup CA,
\]
where~$X \to \wh{X}$ is the quotient map, leaves~$X \setminus
\overline{B}$ fixed and induces an epimorphism between the first
homology groups.  Hence,
$$
b_1(Y) \leq b_1(Y \cup CA) \leq b_1(X).
$$

Now, suppose that the projection of some arc~$\alpha$ of~$X \setminus
B$ to~$Y$ connects two points of~$A$.  Then the space~$Y \cup CA$ is
homotopy equivalent to~$(Y \cup CA') \vee S^{1}$, where~$A'
\subset A$.  That is,
$$
Y \cup CA \simeq (Y \cup CA') \vee S^{1}.
$$
We deduce that
$$
b_1(Y) < b_1(Y \cup CA) \leq b_1(X). 
$$
\forgotten


\begin{lemma} \label{lem:class}
If~$r < \frac{1}{2} \sys(X,\phi)$, then~$Y$ is~$\psi$--essential for
some homomorphism~$\psi \co \pi_{1}(Y) \to G$ such that
\begin{equation} \label{eq:circ}
\psi \circ \pi_* =\phi
\end{equation}
where~$\pi_*$ is the~$\pi_1$--homomorphism induced by \mbox{$\pi \co X \to Y$}.
\end{lemma}

\begin{proof}
Consider the CW-pair~$(X,B)$ where~$B=B(x_0,r)$.  By
Lemma~\ref{lem:trivial}, the restriction of the classifying
map~$\varphi \co X \to K(G,1)$ induced by~$\phi$ to~$B$ is homotopic to a
constant map.  Thus, the classifying map~$\varphi$ extends to~$X \cup
CB$ and splits into
\[
X \hookrightarrow X \cup CB \to K(G,1),
\]
where~$CB$ is a cone over~$B \subset X$ and the first map is the
inclusion map.  Since~$X \cup CB$ is homotopy equivalent to the
quotient~$X/B$, \cf~\cite[Example~0.13]{Hat}, we obtain the following
decomposition of~$\varphi$ up to homotopy:
\begin{equation} \label{eq:XB}
X \stackrel{\pi}{\longrightarrow} Y \to X/B \to K(G,1).
\end{equation}

Hence,~$\psi \circ \pi_* = \phi$ for the~$\pi_1$--homomorphism~$\psi \co \pi_1(Y) \to G$ induced by the map~$Y \to K(G,1)$.  
If the map~$Y \to K(G,1)$ can be homotoped into the~$1$--skeleton of~$K(G,1)$, the same
is true for
\[
X \to Y \to K(G,1)
\]
and so for the homotopy equivalent map~$\varphi$, which contradicts
the~$\phi$--essentialness of~$X$.
\end{proof}

\section{Exploiting a ``fat" ball}

We normalize the~$\phi$--relative systole of~$X$ to one, i.e.~$\sys(X,\phi)=1$.  
Choose a fixed~$\delta \in (0,\frac{1}{2})$
(close to~$0$) and a real parameter~$\lambda > \frac{1}{2}$ (close
to~$\frac{1}{2}$).

\begin{proposition} 
\label{prop:reeb}
Suppose there exist a point~$x_{0} \in X$ and a value~$r_{0} \in
(\delta,\frac{1}{2})$ regular for~$f$ such that
\begin{equation} 
\label{eq:lambda}
\area \, B > \lambda \, (\length \, S)^{2}
\end{equation}
where~$B=B(x_{0},r_{0})$ and~$S=S(x_{0},r_{0})$. 
Then there exists a
piecewise flat metric on~$Y=Y(x_{0},r_{0})$
such that the systolic areas (\cf~Definition~\ref{def:sigma}) satisfy
$$
\sigma_{\psi}(Y) \leq \sigma_{\phi}(X).
$$
\end{proposition}

\begin{proof}
Consider the metric on~$Y$ described in Definition~\ref{def:Y}.
Strictly speaking, the metric on~$Y$ is not piecewise flat since the
connected components of~$S$ are collapsed to points, but it can be
approximated by piecewise flat metrics.

Due to the presence of the buffer cylinders, every loop of~$Y$ of
length less than~$\sys(X,\phi)$ can be deformed into a loop of~$X
\setminus B$ without increasing its length.  Thus, by~\eqref{eq:circ},
one obtains
\begin{equation*}
\sys(Y,\psi) \geq \sys(X,\phi) = 1.
\end{equation*}
Furthermore, we have
\begin{equation*}
\area \, Y \leq \area \, X - \area \, B + \tfrac{1}{2} (\length \,
S)^{2}.
\end{equation*}
Combined with the inequality~\eqref{eq:lambda}, this leads to
\begin{equation} \label{eq:wh}
\sigma_{\psi}(Y) < \sigma_{\phi}(X) - \left( \lambda - \tfrac{1}{2}
\right) (\length \, S)^{2}.
\end{equation}
Hence,~$\sigma_{\psi}(Y) \leq \sigma_{\phi}(X)$,
since~$\lambda > \frac{1}{2}$.
\end{proof}

\section{An integration by separation of variables}

Let~$X$ be a piecewise Riemannian finite connected~$2$--complex.  Let
\mbox{$\phi \co \pi_{1}(X) \to G$} be a nontrivial homomorphism to a
group~$G$.  We normalize the metric to unit relative systole:
$\sys(X,\phi)=1$.  The following area lower bound appeared
in~\cite[Lemma~7.3]{RS08}.

\begin{lemma} \label{lem:BS}
Let~$x \in X$,~$\lambda >0$ and~$\delta \in (0,\frac{1}{2})$.
If 
\begin{equation} \label{eq:BS}
\area \, B(x,r) \leq \lambda \, (\length \, S(x,r))^{2}
\end{equation}
for almost every~$r \in (\delta,\frac{1}{2})$, then 
$$
\area \, B(x,r) \geq \frac{1}{4\lambda} (r-\delta)^{2}
$$
for every~$r \in (\delta,\frac{1}{2})$.

In particular,~$\displaystyle \area(X) \geq \frac{1}{16 \lambda} \,
\sys(X,\phi)^{2}$.
\end{lemma}

\begin{proof}
By the coarea formula, we have
\begin{equation*}
a(r) := \area \, B(x,r) = \int_0^r \ell(s) \, ds
\end{equation*}
where~$\ell(s)=\length \, S(x,s)$.  Since the function~$\ell(r)$ is
piecewise continuous, the function~$a(r)$ is continuously
differentiable for all but finitely many~$r$ in~$(0,\frac{1}{2})$
and~$a'(r)=\ell(r)$ for all but finitely many~$r$
in~$(0,\frac{1}{2})$.  By hypothesis, we have
$$
a(r) \leq \lambda \, a'(r)^2
$$
for all but finitely many~$r$ in~$(\delta,\frac{1}{2})$.
That is,
$$ \left( \sqrt{a(r)} \right)' = \frac{a'(r)}{2 \sqrt{a(r)}} \geq
\frac{1}{2\sqrt{\lambda}}.~$$
We now integrate this differential inequality from~$\delta$ to~$r$, to
obtain
$$
 \sqrt{a(r)} \geq \frac{1}{2\sqrt{\lambda}} (r-\delta).
$$
Hence, for every~$r \in (\delta, \frac{1}{2})$, we obtain
\[
 a(r) \geq \frac{1}{4 \lambda} (r-\delta)^{2},
\]
completing the proof.
\end{proof}
\forget
, or
\begin{equation*}
dr \leq \lambda^{1/2} a^{-1/2} da.
\end{equation*}
We now integrate this differential inequality from~$\delta$ to~$r$, to
obtain
\begin{equation*}
r-\delta \leq \int_{a(\delta)}^{a(r)} \lambda^{1/2} a^{-1/2} da,
\end{equation*}
and hence
\begin{equation*}
r-\delta \leq 2 \lambda^{1/2} \left( a(r)^{1/2} - a(\delta^{1/2})
\right) \leq 2 \lambda^{1/2} a(r)^{1/2},
\end{equation*}
proving the result so long as we have the inequality for all the
intermediate values of~$r$.  

Why is that?
\forgotten

\section{Proof of relative systolic inequality}

We prove that if~$X$ is a~$\phi$--essential piecewise
Riemannian~$2$--complex which is almost minimal (up to~$\varepsilon$),
and has least first Betti number among such complexes, then~$X$ possesses
an~$r$--ball of large area for each~$r< \tfrac{1}{2} \sys(X, \phi)$.
We have not been able to find such a ball for an
arbitrary~$\phi$--essential complex (without the assumption of almost
minimality), but at any rate the area lower bound for almost minimal
complexes suffices to prove the~$\phi$--systolic inequality for
all~$\phi$--essential complexes, as shown below.

\forget
\begin{definition}
Let~$G$ be a group.  We set
\begin{equation*}
\sigma_{*}(G) = \inf_{X} \sigma_{\phi}(X),
\end{equation*}
where the infimum is over all~$\phi$--essential piecewise Riemannian
finite~$2$--complexes~$X$, where the homomorphism~$\phi$ has values
in~$G$.
\end{definition}
\forgotten

\begin{remark}
We do not assume at this point that~$\sigma_{*}(G)$ is nonzero,
\cf~Definition~\ref{def:sigma}.  In fact, the proof of
$\sigma_{*}(G)>0$ does not seem to be any easier than the explicit
bound of Corollary~\ref{coro:A}.
\end{remark}

Theorem~\ref{13} and Corollary~\ref{coro:A} are consequences of the
following result.

\begin{proposition} 
\label{prop:minB}
Let~$\varepsilon >0$.  Suppose~$X$ has a minimal first Betti number
among all~$\phi$--essential piecewise Riemannian~$2$--complexes
satisfying 
\begin{equation} \label{eq:eps}
\sigma_{\phi}(X) \leq \sigma_*(G) +\varepsilon.
\end{equation}  
Then each ball centered at a point~$x$ on a~$\phi$--systolic loop in~$X$
satisfies the area lower bound
\begin{equation*}
\area \, B(x,r) \geq
\frac{(r-\delta)^2}{2+\frac{\varepsilon}{\delta^2}}
\end{equation*}
for every~$r \in \left(\delta,\frac{1}{2}\sys(X,\phi) \right)$, where
$\delta \in \left(0,\frac{1}{2}\sys(X,\phi)\right)$.  In particular,
we obtain the bound
\begin{equation*}
\sigma_*(G) \geq \frac{1}{8}.
\end{equation*}
%
%
\end{proposition}

\forget
\begin{proof}
If for each~$r$ we have~$a(r) \leq a'(r)$ then we separate variables
as in the previous section to obtain the area lower bound.

If for some~$r$, we have~$a(r) > a'(r)$, then there are two
possibilities: either~$S(r)$ is connected, and we obtain a lower bound
of~$r^2$ for the area by Hebda's trick, or~$S(r)$ is disconnected.
But the latter case is impossible by the hypothesis if minimality of
Betti number.
\end{proof}

We can now proceed with the proof of the relative systolic inequality
for essential~$2$--complexes.
\forgotten

\begin{proof}
We will use the notation and results of the previous sections.
Choose~$\lambda > 0$ such that
\begin{equation}
\label{52}
\varepsilon < 4 \left(\lambda - \tfrac{1}{2} \right) \delta^{2}.
\end{equation}
That is,
\[
\lambda > \frac{1}{2} + \frac{\varepsilon}{4 \delta^2} \quad \mbox{
(close to } \frac{1}{2} + \frac{\varepsilon}{4 \delta^2}).
\]
We normalize the metric on~$X$ so that its~$\phi$--systole is equal to one.
Choose a point~$x_{0} \in X$ on a~$\phi$--systolic loop~$\gamma$ of~$X$.  

If the balls centered at~$x_0$ are too ``thin'',
i.e., the inequality~\eqref{eq:BS} is satisfied for~$x_{0}$ and almost
every~$r \in (\delta,\frac{1}{2})$, then the result follows from
Lemma~\ref{lem:BS}.  

We can therefore assume that there exists a ``fat'' ball centered at~$x_0$, i.e., the hypothesis of Proposition~\ref{prop:reeb} holds
for~$x_{0}$ and some regular~$f$--value~$r_{0} \in
(\delta,\frac{1}{2})$, where~$f$ is the distance function from~$x_0$.
(Indeed, almost every~$r$ is regular for~$f$.)
Arguing by contradiction, we show that the assumption on the minimality of the first Betti number rules out this case.

We would like to construct a~$\psi$--essential piecewise
flat~$2$--complex~$Y'$ with~$b_1(Y') < b_1(X)$ such that
$\sigma_{\psi}(Y') \leq \sigma_{\phi}(X)$ and therefore
\begin{equation}
\sigma_{\psi}(Y') \leq \sigma_{*}(G) + \varepsilon
\end{equation}
for some homomorphism~$\psi \co \pi_{1}(Y') \to G$.

By Lemma~\ref{lem:class} and Proposition~\ref{prop:reeb}, the space~$Y = Y(x_{0},r_{0})$, endowed with the piecewise Riemannian metric of Proposition~\ref{prop:reeb}, satisfies
\begin{equation*}
\sigma_{*}(G) \leq \sigma_{\psi}(Y) \leq \sigma_{\phi}(X).
\end{equation*}
Combined with the inequalities~\eqref{eq:wh} in the proof of
Proposition~\ref{prop:reeb} and~\eqref{eq:eps}, this yields
\begin{equation*}
\left( \lambda - \frac{1}{2} \right) (\length \, S)^{2} < \varepsilon.
\end{equation*}
From~$\varepsilon < 4 (\lambda - \frac{1}{2}) \delta^{2}$
and~$\delta \leq r_{0}$, we deduce that
$$
\length \, S < 2 r_{0}.
$$

Now, by Lemma~\ref{lem:trivial}, the~$\phi$--systolic
loop~$\gamma\subset X$ does not entirely lie in~$B$.  Therefore, there
exists an arc~$\alpha_0$ of~$\gamma$ passing through~$x_{0}$ and lying
in~$B$ with endpoints in~$S$.  We have
\[
\length(\alpha_0) \geq 2r_{0}.
\]
If the endpoints of~$\alpha_0$ lie in the same connected component
of~$S$, then we can join them by an arc~$\alpha_1 \subset S$ of length
less than~$2r_{0}$.  By Lemma~\ref{lem:trivial}, the loop~$\alpha_0
\cup \alpha_1$, lying in~$B$, is~$\phi$--contractible.  Therefore, the
loop~$\alpha_1 \cup (\gamma \setminus \alpha_0)$, which is shorter
than~$\gamma$, is~$\phi$--noncontractible.  Hence a contradiction.

This shows that the~$\phi$--systolic loop~$\gamma$ of~$X$ meets two
connected components of~$S$. 

Since a~$\phi$--systolic loop is length-minimizing, the loop~$\gamma$
intersects~$S$ exactly twice.  Therefore, the complementary
arc~$\alpha=\gamma \setminus \alpha_0$, joining two connected
components of~$S$, lies in~$X \setminus B$.
%
%
The two endpoints of~$\alpha$ are connected by a length-minimizing arc
of~$Y \setminus (X \setminus \overline{B})$ passing exactly through
two edges of the cone~$CA$.

Let~$Y'$ be the~$2$--complex obtained by removing the interior of one
of these two edges from~$Y$.  The complex~$Y'=Y \setminus e$ is
clearly connected and the space~$Y$, obtained by gluing back the
edge~$e$ to~$Y$, is homotopy equivalent to~$Y' \vee S^1$.  That is,
\begin{equation} \label{eq:Y'}
Y \simeq Y' \vee S^1.
\end{equation}
Thus,~$Y'$ is~$\psi$--essential if we still denote by~$\psi$ the
restriction of the homomorphism~$\psi \co \pi_1(Y) \to G$
to~$\pi_1(Y')$.  Furthermore, we clearly have
\[
\sigma_{\psi}(Y') = \sigma_{\psi}(Y) \leq \sigma_{\phi}(X).
\]
Combined with~\eqref{eq:b1}, the homotopy equivalence~\eqref{eq:Y'}
also implies
$$
b_1(Y') < b_1(Y) \leq b_1(X).
$$
Hence the result.
\end{proof}

\begin{remark}
We could use round metrics (of constant positive Gaussian curvature)
on the ``buffer cylinders" of the space~$Y$ in the proof of
Proposition~\ref{prop:reeb}.  This would allow us to choose~$\lambda$
close to~$\frac{1}{2 \pi}$ and to derive the lower bound
of~$\frac{\pi}{8}$ for~$\sigma_{\phi}(X)$ in Corollary~\ref{coro:A}.
We chose to use flat metrics for the sake of simplicity.
\end{remark}

\section{Cohomology of Lens spaces}

Let~$p$ be a prime number.  The group~$G=\Z_p$ acts freely on the
contractible sphere~$S^{2\infty+1}$ yielding a model for the
classifying space
\begin{equation*}
K = K(\Z_{p},1) = S^{2\infty+1}/\Z_{p}.
\end{equation*}
The following facts are well-known, \cf~\cite{Hat}.

\begin{proposition}
\label{42}

The cohomology ring~$H^*(\Z_p;\Z_p)$ for~$p$ an odd prime is the
algebra~$\Z_p(\alpha)[\beta]$ which is exterior on one
generator~$\alpha$ of degree~$1$, and polynomial with one
generator~$\beta$ of degree~$2$.  Thus,
\begin{itemize}
\item
$\alpha$ is a generator of~$H^1(\Z_p;\Z_p)\simeq \Z_p$,
satisfying~$\alpha^2=0$;
\item
$\beta$ is a generator of~$H^2(\Z_p;\Z_p)\simeq \Z_p$.
\end{itemize}
\end{proposition}

Here the~$2$--dimensional class is the image under the Bockstein
homomorphism of the~$1$--dimensional class.  The cohomology of the
cyclic group is generated by these two classes.  The cohomology is
periodic with period~$2$ by Tate's theorem.  Every even-dimensional
class is proportional to~$\beta^n$.  Every odd-dimensional class is
proportional to~$\alpha \cup \beta^n$.

Furthermore, the reduced integral homology is~$\Z_p$ in odd dimensions
and vanishes in even dimensions.  The integral cohomology is~$\Z_p$ in
even positive dimensions, generated by a lift of the class~$\beta$
above to~$H^2(\Z_p;\Z)$.

\begin{proposition}
\label{41}
\label{33}
Let~$M$ be a closed~$3$--manifold~$M$ with~$\pi_1(M)=\Z_{p}$.  Then its
classifying map~$\varphi \co M \to K$ induces an
isomorphism
\[
\varphi_i \co H_i(M;\Z_p)\simeq H_i(K;\Z_p)
\]
for~$i=1,2,3$.
\end{proposition}

\begin{proof}
Since~$M$ is covered by the sphere, for~$i=2$ the isomorphism is a
special case of Whitehead's theorem.  Now consider the exact sequence
(of Hopf type)
\begin{equation*}
\pi_3(M) \overset{\times p}{\longrightarrow}H_3(M;\Z)\to
H_3(\Z_p;\Z)\to 0
\end{equation*}
since~$\pi_2(M)=0$.  Since the homomorphism~$H_3(M;\Z) \to
H_3(\Z_p;\Z)$ is onto, the result follows by reduction modulo~$p$.
\end{proof}

\section{Volume of a ball}

Our Theorem \ref{theo:main} is a consequence of the following result.

\begin{theorem} 
\label{theo:ball}
Assume the~$\rm{GG}_C$-property~\eqref{eq:ball} is satisfied for some universal constant
$C>0$ and every homomorphism~$\phi$ into a finite group~$G$. 
Then every closed
Riemannian~$3$--manifold~$M$ with fundamental group~$G$ contains a
metric ball~$B(R)$ of radius~$R$ satisfying
\begin{equation}
\label{24}
\vol \, B(R) \geq \frac{C}{3} R^3,
\end{equation}
for every~$R\leq\frac{1}{2}\sys(M)$.
\end{theorem}

\forget
Recall the following result.

\begin{proposition}
\label{006}
In an orientable~$3$--manifold, cup product on~$H^1\otimes H^2$ in
cohomology with~$\Z_p$ coefficients is dual to intersection between
a~$2$--cycle and a ~$1$--cycle with coefficients in~$\Z_p$.
\end{proposition}

Here the global orientation allows one to count an integer
intersection index, which is then reduced modulo~$p$. \\
\forgotten

We will first prove Theorem~\ref{theo:ball} for a
closed~$3$--manifold~$M$ of fundamental group~$\Z_{p}$, with~$p$
prime.  We assume that~$p$ is odd (the case~$p=2$ was treated by
L.~Guth).  In particular,~$M$ is orientable.  Let~$D$ be a~$2$--cycle
representing a nonzero class~$[D]$ in
\begin{equation*}
H_2(M;\Z_p) \simeq H_{1}(M;\Z_{p}) \simeq \Z_p.
\end{equation*}
Denote by~$D_0$ the finite~$2$--complex of~$M$ given by the support
of~$D$.  Without loss of generality, we can assume that~$D_0$ is
connected.  The restriction of the classifying map~$\varphi \co M \to
K$ to~$D_0$ induces a homomorphism~$\phi \co \pi_{1}(D_0) \to \Z_{p}$.
%
%

\begin{lemma} \label{lem:DB}
The cycle~$D$ induces a trivial relative class in the homology of
every metric~$R$--ball~$B$ in~$M$ relative to its boundary, with~$R <
\frac{1}{2} \sys(M)$.  That is,
$$
[D \cap B] = 0 \in H_{2}(B,\partial B;\Z_{p}).
$$
\end{lemma}

\begin{proof}
Suppose the contrary.  By the Lefschetz-Poincar\'e duality theorem,
the relative~$2$--cycle~$D \cap B$ in~$B$ has a nonzero intersection
with an (absolute)~$1$--cycle~$c$ of~$B$.  Thus, the intersection
between the~$2$--cycle~$D$ and the~$1$--cycle~$c$ is nontrivial
in~$M$.  Now, by Lemma~\ref{lem:trivial}, the~$1$--cycle~$c$ is
homotopically trivial in~$M$.  Hence a contradiction.
\end{proof}

We will exploit the following notion of volume for cycles with torsion
coefficients.

\begin{definition} 
\label{def:Vol}
Let~$D$ be a~$k$--cycle with coefficients in~$\Z_p$ in a Riemannian
manifold~$M$.  We have
\begin{equation}
\label{11}
D= \sum_i n_i \sigma_i
\end{equation}
where each~$\sigma_i$ is a~$k$--simplex, and each~$n_i\in \Z_p^*$ is
assumed nonzero.  We define the notion of~$k$--area~$\area$ for cycles
as in \eqref{11} by setting
\begin{equation}
\label{12}
\area(D)= \sum_i |\sigma_i|,
\end{equation}
where~$|\sigma_i|$ is the~$k$--area induced by the Riemannian metric
of~$M$.
\end{definition}

\begin{remark}
The non-zero coefficients~$n_i$ in \eqref{11} are ignored in defining
this notion of volume.
\end{remark}

\begin{proof}[Proof of Theorem~\ref{theo:ball}]
We continue the proof of Theorem~\ref{theo:ball} when the fundamental
group of~$M$ is isomorphic to~$\Z_p$, with~$p$ an odd prime.  We will
use the notation introduced earlier.  Suppose now that~$D$ is a
piecewise smooth~$2$--cycle area minimizing in its homology
class~$[D]\not=0\in H_2(M;\Z_p)$ up to an arbitrarily small error
term~$\varepsilon>0$, for the notion of volume (area) as defined
in~\eqref{12}.


Recall that~$\phi \co \pi_1(D_0)\to \Z_p$ is the homomorphism induced
by the restriction of the classifying map~$\varphi \co K \to M$ to the
support~$D_0$ of~$D$.  By Proposition~\ref{33}, the~$2$--complex~$D_0$
is~$\phi$--essential.  Thus, by hypothesis of Theorem~\ref{theo:ball},
we can choose a point~$x \in D_0$ satisfying
the~$\rm{GG}_C$-property~\eqref{eq:ball}, i.e., the area of~$R$--balls
in~$D_0$ centered at~$x$ grows at least as~$C R^2$ for~$R <
\frac{1}{2} \sys(D_0,\phi)$.
%
%
Therefore, the intersection of~$D_0$ with the~$R$--balls of$M$
centered at~$x$ satisfies
\begin{equation}
\label{111}
\area(D_0\cap B(x,R)) \geq CR^2
\end{equation}
for every~$R < \frac{1}{2} \sys(D_0,\phi)$.
The idea of the proof is to control the area of distance spheres
(level surfaces of the distance function) in~$M$, in terms of the
areas of the distance disks in~$D_0$.

Let~$B=B(x,R)$ be the metric~$R$--ball in~$M$ centered at~$x$
with~$R<\frac{1}{2} \sys(M)$.  We subdivide and slightly perturb~$D$
first, to make sure that~$D \cap \bar B$ is a subchain of~$D$.  Write
\[
D=D_- + D_+,
\]
where~$D_-$ is a relative~$2$--cycle of~$\bar B$, and~$D_+$ is a
relative~$2$--cycle of~$M\setminus B$.  By Lemma~\ref{lem:DB},~$D_-$
is homologous to a~$2$--chain~$\mathcal{C}$ contained in the distance
sphere~$\partial B = S(x,R)$ with
\[
\partial \mathcal{C} = \partial D_- = - \partial D_+.
\]
We subdivide and perturb~$\mathcal{C}$ in~$S(x,R)$ so that the
interiors of its~$2$--simplices either agree or have an empty
intersection.  Here the simplices of the~$2$--chain~$\mathcal{C}$ may
have nontrivial multiplicities.
%
%
Such multiplicities necessarily affect the volume of a chain if one
works with integer coefficients.
%
%
However, these multiplicities are ignored for the notion
of~$2$--volume~\eqref{12}.  This special feature allows us to derive
the following: the~$2$--volume~\eqref{12} of the chain~$\mathcal{C}$
is a lower bound for the usual area of the distance sphere~$S(x,R)$.

Note that the homology class~$[\mathcal{C}+D_+]=[D] \in H_2(M;\Z_p)$
stays the same.  We chose~$D$ to be area minimizing up
to~$\varepsilon$ in its homology class in~$M$ for the notion of
volume~\eqref{12}.  Hence we have the following bound:
\begin{equation}
\label{112}
\area(S(x,R)) \geq \area(\mathcal{C}) \geq \area(D_-) - \varepsilon
\geq \area (D_0 \cap B) - \varepsilon.
\end{equation}
Now, clearly~$\sys(M) \leq \sys(D_0,\phi)$.  Combining the
estimates~\eqref{111} and~\eqref{112}, we obtain
\begin{equation}
\label{113}
\area ( S(x,R)) \geq C R^2 - \varepsilon
\end{equation}
for every~$R<\frac{1}{2} \sys(M)$.  Integrating the
estimate~\eqref{113} with respect to~$R$ and letting~$\varepsilon$ go
to zero, we obtain a lower bound of~$\frac{C}{3} R^3$ for
the~$3$--volume of some~$R$--ball in the closed manifold~$M$, proving
Theorem~\ref{theo:ball} for closed~$3$--manifolds with fundamental
group~$\Z_{p}$. \\

Suppose now that~$M$ is a closed~$3$--manifold with finite (nontrivial)
fundamental group.  Choose a prime~$p$ dividing the
order~$|\pi_1(M)|$ and consider a cover~$N$ of~$M$ with fundamental group cyclic
of order~$p$.
This cover satisfies~$\sys(N) \geq \sys(M)$, and we apply the
previous argument to~$N$.

Note that the reduction to a cover could not have been done in the
context of M.~Gromov's formulation of the inequality in terms of the
global volume of the manifold.  Meanwhile, in our formulation using a
metric ball, following L.~Guth, we can project injectively the ball of
sufficient volume, from the cover to the original manifold.  Namely,
the proof above exhibits a point~$x \in N$ such that the volume of
the~$R$--ball~$B(x,R)$ centered at~$x$ is at least~$\frac{C}{3} R^3$
for every~$R < \frac{1}{2} \sys(M)$.  Since~$R$ is less than half the
systole of~$M$, the ball~$B(x,R)$ of~$N$ projects injectively to an
$R$--ball in~$M$ of the required volume, completing the proof of
Theorem~\ref{theo:ball}.
\end{proof}



 \end{document}